\documentclass[a4paper,10pt]{article}
\usepackage{a4}
\usepackage{amsmath,amsthm,amssymb,enumerate}
\usepackage{amsthm}

\usepackage{amsfonts, amsmath}   
\usepackage{amssymb}
\usepackage{amsmath}
\usepackage{graphics}

\usepackage{amsmath,amsthm,amssymb,enumerate}
\usepackage{amsthm}
\newtheorem{theorem}{Theorem}[section]

\newtheorem{prop}[theorem]{Proposition}
\newtheorem{prb}[theorem]{Problem}
\newtheorem{obs}[theorem]{Observation}
\newtheorem{cor}[theorem]{Corollary}
\newtheorem{dfn}[theorem]{Definition}
\newtheorem{exm}[theorem]{Example}

\newtheorem{con}[theorem]{Conjecture}
\newtheorem{defn}[theorem]{Definition}
\numberwithin{equation}{section}

\begin{document}

\title{\textbf{\ Secure Domination in Digraphs} }
\author{ {\bf \sc  Mart\'{\i}n Manrique$^{1},$   Karam Ebadi$^1$ and Morteza Ebrahimi$^{2}$ }\\{\footnotesize $^1$National Centre for Advanced Research in Discrete Mathematics ($n$-CARDMATH)} \\{\footnotesize  Kalasalingam University}\\{\footnotesize Anand Nagar,  Krishnankoil-626126, India.}\\   {\footnotesize  $^2$ Department of Mathematics Islamic Azad University of Shahindej, Iran.}\\
{\footnotesize e-mails: {\it martin.manrique@gmail.com, karam\_ebadi@yahoo.com, M.ebrahimi55@gmail.com}}}

\date{}

\maketitle

\begin{abstract}
Given a graph $G=(V,E),$ a set $S\subseteq V$ is dominating if for
every $v\in V\setminus S$ there exists $u\in S$ such that $uv\in E.$ A
dominating set $S\subseteq V$ is secure if for every $v\in V\setminus
S$ there exists $u\in S$ such that $(S\setminus \{u\})\cup \{v\}$ is
dominating. In this work we extend the concept of secure dominating set to
digraphs in four different ways, all of them with interesting applications,
and prove some results regarding each of them.
\end{abstract}



\section{Introduction}

Throughout this paper $D=(V,A)$ is a finite directed graph with neither
loops nor multiple arcs (but pairs of opposite arcs are allowed) and $%
G=(V,E) $ is a finite undirected graph with neither loops nor multiple
edges. Unless stated otherwise, $n$ denotes the order os $D$ (or $G$). For basic terminology on graphs and digraphs, we refer to Chartrand
and Lesniak \cite{cl}.

Let $D=(V,A)$ be a digraph. For any vertex $v\in V$, the sets $%
N^{+}(u)=\{v: uv\in A\}$ and $N^{-}(u)=\{v: vu\in A\}$ are called the
out-neighborhood and in-neighborhood of $u,$ respectively. $N^{+}[u]=N^{+}(u)\cup \{u\}$ is the closed out-neighborhood of $u,$ and $N^{-}[u]=N^{-}(u)\cup \{u\}$ is the closed in-neighborhood of $u.$ The in-degree and out-degree of $u$ are defined by $%
d^{-}(u)=|N^{-}(u)|$ and $d^{+}(u)=|N^{+}(u)|$. The minimum in-degree, the
minimum out-degree, the maximum in-degree and the maximum out-degree of $D$ are
denoted by $\delta ^{-}$, $\delta ^{+}$, $\Delta ^{-}$ and $\Delta ^{+}$
respectively, while $\delta ^{0}=\min\{\delta ^{-},\delta ^{+}\}$ is the minimum degree of $D.$

Let $G=(V,E)$ be a graph. A subset $S$ of $V$ is called a dominating set of $%
G$ if every vertex in $V\setminus S$ is adjacent to at least one vertex in $%
S.$ The minimum cardinality of a dominating set of $G$ is called the
domination number of $G$ and is denoted by $\gamma (G)$ or simply $\gamma .$

Let $D=(V,A)$ be a digraph. A subset $S$ of $V$ is called an out-dominating
set of $D$ if for every vertex $v\in V\setminus S$ there exists at least one
vertex $u\in S\cap N^{-}(v).$ The minimum cardinality of an out-dominating
set of $D$ is called the out-domination number of $D$ and is denoted by $%
\gamma ^{+}(D)$, or simply $\gamma ^{+}.$ In-dominating sets in digraphs are
defined in a similar way, and the minimum cardinality of an in-dominating
set of $D$ is called the in-domination number of $D,$ denoted by $%
\gamma ^{-}(D)$.

Although domination and other related concepts have been extensively studied
for undirected graphs, the respective analogues on digraphs have not received
much attention. Fu \cite{fu} studied the out-domination number of a directed
graph $D=(V,A)$. Arumugam et al. \cite{sa} introduced the concepts of total
and connected domination in digraphs.

A survey of results on domination in directed graphs by Ghoshal, Laskar and
Pillone is found in chapter 15 of Haynes et al. \cite{h}, but most of the
results in this survey deal with the concepts of kernels and solutions (that
is, independent in- and out-dominating sets) in digraphs and on domination
in tournaments.

Given an undirected graph $G=(V,E)$, the set $S\subseteq V$ is a secure
dominating set (SDS) of $G$ if it is dominating and for each $u\in V\setminus S$ there exists $%
v\in N(u)\cap S$ such that $(S\setminus \{v\}\cup \{u\}$ is a dominating
set. The minimum cardinality of an SDS of $G$ is called the secure domination
number of $G$ and is denoted by $\gamma _{s}(G),$ while a minimum SDS is called a $\gamma _{s}$-set \cite{E3,  ma, wf}.

This notion can be extended to digraphs in several ways. There are three
very natural extensions of the concept:

\begin{dfn}
Let $D=(V,A)$ be a digraph. A subset $S\subseteq V$ is called a secure
out-dominating set (SODS) of $D$ if $S$ is out-dominating and for every
vertex $v\in V\setminus S$, there exists a vertex $u\in (N^{+}(v)\cup N^{-}(v))\cap S$ such that $(S\setminus \{u\})\cup \{v\}$ is an
out-dominating set. In this case we say that $u$ defends $v$. The minimum
cardinality of an SODS in $D$ is called the secure
out-domination number of $D$ and is denoted by $\gamma _{so}(D),$ while a minimum SODS is called a $\gamma _{so}$-set.
\end{dfn}

\begin{dfn}
Let $D=(V,A)$ be a digraph and let $G$ be the underlying undirected graph of
$D$. A subset $S\subseteq V$ is called an out-secure dominating set (OSDS)
of $D$ if $S$ is dominating in $G$ and for every vertex $v\in V\setminus S$,
there exist a vertex $u\in N^{-}(v)\cap S$ such that $(S\setminus \{u\})\cup
\{v\}$ is a dominating set of $G$. In this case we say that $u$ defends $v$.
The minimum cardinality of an OSDS in $D$ is called the
out-secure domination number of $D$ and is denoted by $\gamma _{os}(D),$ while a minimum OSDS is called a $\gamma _{os}$-set.
\end{dfn}

\begin{dfn}
Let $D=(V,A)$ be a digraph. A subset $S\subseteq V$ is called an out-secure
out-dominating set (OSODS) of $D$ if $S$ is out-dominating and for every
vertex $v\in V\setminus S$, there exist a vertex $u\in N^{-}(v)\cap S$ such
that $(S\setminus \{u\})\cup \{v\}$ is an out-dominating set. In this case
we say that $u$ defends $v$. The minimum cardinality of an OSODS in $D$ is called the out-secure out-domination number of $%
D$ and is denoted by $\gamma _{oso}(D),$ while a minimum OSODS is called a $\gamma _{oso}$-set.
\end{dfn}

Of course, these three concepts can be defined as well for in-dominating and
in-secure sets. However, as happens with solutions and kernels, a result in
the out- version for a digraph $D=(V,A)$ corresponds to a result in the in-
version for $\overleftarrow{D}=(V,\overleftarrow{A})$, where $\overleftarrow{%
A}=\{uv:vu\in A\}$. Therefore, the study of the whole matter can be accomplished by
choosing only the out- or the in- version.

Another way of extending secure dominating sets to digraphs is the following:

\begin{dfn}
Let $D=(V,A)$ be a digraph. A subset $S\subseteq V$ is called an in-secure
out-dominating set (ISODS) of $D$ if $S$ is out-dominating and for every
vertex $v\in V\setminus S$, there exist a vertex $u\in N^{+}(v)\cap S$ such
that $(S\setminus \{u\})\cup \{v\}$ is an out-dominating set. In this case
we say that $u$ defends $v$. The minimum cardinality of an ISODS in $D$ is called the in-secure out-domination number of $%
D $ and is denoted by $\gamma _{iso}(D),$ while a minimum ISODS is called a $\gamma _{iso}$-set.
\end{dfn}

Of course, we can talk of out-secure in-dominating sets, but any result regarding them for a digraph $D$ will correspond to a result on in-secure out-dominating
sets for $\overleftarrow{D}$.

\begin{exm} The digraph  $D$ given in Figure 1 is an example where $\gamma^{+}=2,$  $\gamma_{os}=2,$ $\gamma_{so}=3,$ $\gamma_{oso}=4,$ and $\gamma_{iso}=5:$ It is easy to check that $\{v_4,v_5\}$ is a minimum out-dominating set, as well as a minimum OSDS;  $\{v_3,v_4,v_5\}$ is a minimum SODS; $\{v_1,v_2,v_4,v_5\}$ is a minimum OSODS, and  $\{v_1,v_2,v_3,v_6,v_7\}$ is a minimum ISODS. There are also simple examples in which $\gamma^{+}<\gamma_{os},$ $\gamma_{so} < \gamma_{os},$ and $\gamma_{iso} < \gamma_{oso}.$\\

\begin{center}
\unitlength .85mm 
\linethickness{0.4pt}
\ifx\plotpoint\undefined\newsavebox{\plotpoint}\fi 
\begin{picture}(53.5,61)(0,0)
\put(36.25,52){\circle*{1.5}}
\put(36.25,60.25){\circle*{1.5}}
\put(49.25,34){\circle*{1.5}}
\put(23.25,34){\circle*{1.5}}
\put(23.25,34.25){\vector(1,0){25.25}}
\put(36.5,44){\circle*{1.5}}
\put(36.25,24){\circle*{1.5}}
\put(36.25,14){\circle*{1.5}}
\put(36.75,24.5){\vector(-4,-3){.07}}\multiput(49.25,34)(-.044326241,-.033687943){282}{\line(-1,0){.044326241}}
\put(36.5,14.5){\vector(-2,-3){.07}}\multiput(49.25,34)(-.033730159,-.051587302){378}{\line(0,-1){.051587302}}
\put(35.75,24.25){\vector(4,-3){.07}}\multiput(23,34)(.044117647,-.033737024){289}{\line(1,0){.044117647}}
\put(35.75,14.5){\vector(2,-3){.07}}\multiput(23,33.75)(.033730159,-.050925926){378}{\line(0,-1){.050925926}}
\put(35.5,51.75){\vector(3,4){.07}}\multiput(23.25,34.75)(.033653846,.046703297){364}{\line(0,1){.046703297}}
\put(35.75,60.25){\vector(1,2){.07}}\multiput(23,35)(.033730159,.066798942){378}{\line(0,1){.066798942}}
\put(48.25,35){\vector(2,-3){.07}}\multiput(36.5,51.75)(.033667622,-.047994269){349}{\line(0,-1){.047994269}}
\put(36.75,43.5){\vector(-4,3){.07}}\multiput(48.25,34.5)(-.043071161,.033707865){267}{\line(-1,0){.043071161}}
\put(36,43.25){\vector(-3,-1){.07}}\multiput(36.75,43.5)(-.09375,-.03125){8}{\line(-1,0){.09375}}
\put(49.5,34.75){\vector(1,-2){.07}}\multiput(36.75,60.5)(.033730159,-.068121693){378}{\line(0,-1){.068121693}}
\put(35,56){$v_1$}
\put(35,47.75){$v_2$}
\put(35,39.75){$v_3$}
\put(17.75,34){$v_4$}
\put(51.5,34){$v_5$}
\put(35,26.75){$v_6$}
\put(35,18){$v_7$}
\put(27.25,8.25){Figure 1:}
\put(24,34.75){\vector(-4,-3){.07}}\multiput(35.5,43.75)(-.043071161,-.033707865){267}{\line(-1,0){.043071161}}
\end{picture}
\end{center}

\end{exm}

The four notions defined above have interest from the mathematical point of
view. Moreover, useful applications for all four can be found. In every
case, we consider our universe as a finite set of vertices, and a set of
elements that must cover it, either protecting, surveying, providing a
service, etc., and which must be promptly helped or replaced if necessary:

Suppose an element in point $v$ provides a service to point $u$, but the
converse not necessarily holds; however, it is not much more difficult to go
from $u$ to $v$ than from $v$ to $u$. Then the situation can be modelled as
an SODS in a digraph. As an example, a warden up in a
hill can survey an adjacent valley, but the converse is not true; if there
is a road and the team has cars, it is not much more difficult nor takes
much more time to go up the hill than down the hill.

If the service can be as easily provided from $u$ to $v$ than from $v$ to $u,$ but transport from $u$ to $v$ is much easier than transport from $v$ to $u,$
 the situation corresponds to an OSDS. For example,
broadcasting towers in different points of a river bank: If the river runs
down a somewhat flat area, a broadcasting tower in $u$ provides service to $%
v $ and vice versa. However, if the current is strong it is much easier to
go down the river than up the river.

When both service and transport are easy in one direction but difficult in
the other, then our set of elements is an OSODS. For
example, a broadcasting tower (or an army) up in the hill covers (protects)
the adjacent valley, but a tower (army) in the valley does not cover (does
not protect) the upper part of the hill. In a similar way, if roads are not
available it is much easier to go down the hill than up the hill.

If service is much easily provided in one direction, but transport is much
easier in the other, then the situation is that of an ISODS. As an example, we have a thick forest up the river and
an open area down the river. A warden or camera in the open area can survey
migratory birds or helicopters passing over itself and over the forest
area, but if he (it) is on the dense vegetation spot it can only detect
those passing over that spot, not those going over the open area. However,
as mentioned above, transport may be much easier down the river than up the
river.

\begin{obs}\label{obs 1.6}
If $D$ is a symmetric digraph, then the notions of SODS, OSDS, OSODS, and
ISODS coincide, and they coincide as well with the concept of SDS in the
underlying undirected graph of $D$.
\end{obs}

\begin{obs} \label{obs 1.7}
Let $D$ be a digraph, and let $D^{\prime }$ be a spanning subdigraph of $D$.
Then $\gamma _{so}(D)\leq \gamma _{so}(D^{\prime })$, $\gamma _{os}(D)\leq
\gamma _{os}(D^{\prime })$, $\gamma _{oso}(D)\leq \gamma _{oso}(D^{\prime })$%
, and $\gamma _{iso}(D)\leq \gamma _{iso}(D^{\prime })$.
\end{obs}

Now we will show the relations between the concepts defined above:

\begin{prop}\label{prop 1.8}
Let $D$ be a digraph. Then we have the following inequality chains (although $\gamma _{os}\leq \gamma _{iso}$ only holds for digraphs without symmetric arcs):
\begin{align*}
\gamma _{s},\ \gamma ^{+}\leq\gamma _{os},\ \gamma _{so}\leq\gamma _{oso},\ \gamma _{iso} .
\end{align*}

\end{prop}

\begin{proof} $\gamma _{s}\leq \gamma _{so}$ and $\gamma _{s}\leq \gamma
_{os}$ follow directly from Observation \ref{obs 1.6} and Observation \ref{obs 1.7}. Since an SODS is
out-dominating, then $\gamma ^{+}\leq \gamma _{so}$. Let $S$ be an OSDS,
then for every $u\in V\setminus S$ there exists $v\in S\cap N^{-}(u)$, that
is, $S$ is out-dominating, and hence $\gamma ^{+}\leq \gamma _{os}$. Now let
$S$ be an OSODS, then $S$ is both an SODS and an OSDS, which implies $\gamma
_{so}\leq \gamma _{oso}$ and $\gamma _{os}\leq \gamma _{oso}$. In a similar
way, every ISODS is an SODS, so $\gamma _{so}\leq \gamma _{iso}$.

Now consider a digraph $D=(V,A)$ without symmetric arcs and let $S$ be an
ISODS of $D.$ For every $v\in V\setminus S,$ there are at least one vertex $%
u\in N^{-}(v)\cap S$ and one vertex $u^{\prime }\in N^{+}(v)\cap S.$ Since $D$
has no symmetric arcs, $u \neq u^{\prime }.$ Then $S^{\prime }=(S\setminus
\{u\})\cup \{v\}$ is a dominating set of $G,$ the underlying graph of $D$: $u
$ is dominated by $v$; every vertex $w\in N^{+}(u)\cap (V\setminus S)$ has
an out-neighbor in $S^{\prime },$ and every vertex $w^{\prime }\in
(V\setminus S)\setminus N^{+}(u)$ has an in-neighbor in $S^{\prime }.$
Therefore, $S$ is an OSDS, which implies $\gamma _{os}\leq \gamma _{iso}.$
\end{proof}

Now we state two observations and two definitions which are useful for the study
of the parameters defined above.

\begin{obs}\label{obs 1.99}
If a vertex $u$ in a digraph $D$ has in-degree $0,$ then $u$ necessarily belong to every out-dominating set. If a vertex $v$ has out-degree $0,$ then $v$ necessarily belongs to every ISODS. An isolated vertex belongs to every OSDS. On the other hand, if $w$ is a vertex of $D$ with an in-neighbor $x$ and an out-neighbor $y,$ then $V(D)\setminus \{y\}$ is an OSODS,  while  $V(D)\setminus \{w\}$ is an ISODS of $D.$ Moreover, for every nontrivial digraph without symmetric arcs, $2\leq \gamma _{so},\ \gamma _{oso},\ \gamma _{iso}.$
\end{obs}

\begin{obs}\label{obs 1.9}
 For the directed path $P_{n}$ with $n\geq 1$ vertices, $\gamma
^{+}(P_{n})=\lceil \frac{n}{2}\rceil $ and for the directed cycle $C_{n}$ with $%
n\geq 3$ vertices, $\gamma ^{+}(C_{n})=\lceil \frac{n}{2}\rceil .$
\end{obs}

\begin{defn}
Let $D=(V,A)$ be a directed graph, $S\subset V$ and $u\in S.$ A vertex $v\in
V$ is called an out-private neighbor of $u$ with respect to $S$ if $%
N^{-}[v]\cap S=\{u\},$ and $v$ is called an in-private neighbor of $u$ with
respect to $S$ if $N^{+}[v] \cap S=\{u\}.$ The set of all out-private
neighbors of $u$ with respect to $S$ is denoted by $pn^{+}(u,S)$ and the set
of all in-private neighbors of $u$ with respect to $S$ is denoted by $%
pn^{-}(u,S).$
\end{defn}

\section{Out-secure out-dominating sets}

\begin{prop} \label{prop 2.1}
Let $S$ be an OSODS  of a digraph  $D.$  A vertex $u\in N^{-}(v)\cap S$ defends   a vertex $v\in V\setminus S$ if, and only if, $ N^{-}(u)\cap S\neq \emptyset$ and
  $ pn^{+}(u,S)\subseteq N^{+}[v].$
\end{prop}

\begin{proof}
  If $N^{-}(u)\cap S=\emptyset ,$ then $(S\setminus \{u\})\cup \{v\}$
does not out-dominate $u$. If there is a vertex $w\in pn^{+}(u,S)\setminus N^{+}[v]$, then $(S\setminus
\{u\})\cup \{v\}$ does not out-dominate $w.$ The converse is obvious.
\end{proof}

\begin{cor} \label{cor 2.2}
A set $S\subseteq V$ is an OSODS  of    $D$   if, and only if, for every $v\in V\setminus S,$ there exists $u\in N^{-}(v)\cap S$ such that   $ N^{-}(u)\cap S\neq \emptyset$ and  $ pn^{+}(u,S)\subseteq N^{+}[v].$
\end{cor}

\begin{theorem}\label{thm 2.3}
Let $D$ be a digraph without symmetric arcs. Then
   $\frac{2n}{2\Delta^{+}+1}\leq  \gamma_{oso}.$
\end{theorem}

\begin{proof}

Let $S$ be a minimum OSODS of $D$ and let $\mu$ be the number of isolated vertices in the induced subdigraph $\langle S \rangle.$ From Proposition \ref{prop 2.1}, those  vertices cannot defend any vertex in $V\setminus S,$  so the isolated vertices in the induced subdigraph $\langle S \rangle$ do not have out-private neighbors. Moreover, the induced subdigraph $\langle S \rangle$ has at least $\lceil \frac{|S|-\mu}2 \rceil$ arcs. Therefore,
$n\leq |S|+\Delta^{+}(|S|-\mu)- \lceil \frac{|S|-\mu}2\rceil\leq |S|+\Delta^{+}|S|-\Delta^{+}\mu-  \frac{|S|}2+\frac{\mu}2\leq|S|+\Delta^{+}|S|-  \frac{|S|}2=  \gamma_{oso}(\Delta^{+}+\frac{1}2).$    Hence,  $\frac{2n}{2\Delta^{+}+1}\leq  \gamma_{oso}.$
 \end{proof}

\begin{theorem} \label{thm 2.4}
 For any digraph $D$, $\gamma_{oso}(D)\leq n-\delta^{0}.$
\end{theorem}

\begin{proof}  The result is trivial if $\delta^{0}=0.$ Hence we assume that $\delta^{0}>0.$ Then for every $v\in V,$ $d^+(v)\geq\delta^0.$ Take $u\in V$ and $B\subseteq N^+(u)$ such that $|B|=\delta^0.$ Then $V\setminus B$ is an OSODS of $D,$ since for every $v\in V,\ d^-(v)\geq \delta^0.$ Therefore, $u$ defends $w$ for every $w\in N^+(u),$ since $N^-(u)\neq \emptyset$ and for every $w\in N^+(u),$ $N^-(w)\cap(V\setminus N^+(u))\neq \emptyset.$
 \end{proof}

\begin{theorem} \label{thm 2.5}
 Let $D$ be  any digraph. Then   $\gamma_{oso}=n$  if, and only if, for every   $u\in V(D),$  $d^{+}(u)=0$ or $d^{-}(u)=0$.
\end{theorem}

\begin{proof}  $\Leftarrow:$ Suppose there exists at least one vertex $v\in D$ such that  $d^{+}(v)>0$ and $d^{-}(v)>0,$ then from Observation \ref{obs 1.99} we have that  $\gamma_{oso}\leq n-1,$ a contradiction.

$\Rightarrow:$ Assume that for every   $u\in V(D),$  $d^{+}(u)=0$ or $d^{-}(u)=0.$ Then obviously all the vertices of in-degree zero must be in every out-dominating set. Let $S$ denote the set of all such vertices. For each $v\in V\setminus S$, every in-neighbor of $v$ has in-degree zero. From Proposition \ref {prop 2.1}, $v$ is not defended, so it must be in every OSODS of $D.$ Hence $\gamma_{oso}=n.$
\end{proof}

\begin{cor} \label{thm 2.6}
 A nontrivial graph $G$ admits an  orientation $D$ such that  $\gamma_{oso}=n$ if, and only if, $G$ is a bipartite graph.
\end{cor}
\begin{proof}
Let $G$ be any graph which has an orientation $D$ such that $\gamma
_{oso}(D)=n.$ If $d^{+}(v)>0$ and $d^{-}(v)>0$ for some $v\in V(D),$ then
Theorem \ref{thm 2.5} implies $\gamma _{oso}(D)\leq n-1,$ a contradiction.
Therefore $d^{+}(v)=0$ or $d^{-}(v)=0$ for all $v\in V(D).$ Suppose $G$ is
not a bipartite graph. Let $C_{2r+1}:(v_{1},v_{2},...,v_{2r+1},v_{1})$ be an
odd cycle in $G.$ Without loss of generality assume that $d^{+}(v_{1})=0,$
then $d^{-}(v_{2})=0,$ $d^{+}(v_{3})=0$ and so on. Hence $d^{+}(v_{i})=0$ if
$i$ is odd and $d^{-}(v_{i})=0$ if $i$ is even for every $v_{i}\in
V(C_{2r+1}), $ it follows that $d^{+}(v_{2r+1})=0,$ which is a contradiction
to $d^{+}(v_{1})=0.$ Therefore $G$ is a bipartite graph.

Conversely, assume that $G$ is a bipartite graph with bipartition $(X,Y)$.
Define the orientation $D$ on $G$ as follows: $d^{+}(v)=0$ for all $v\in X$.
Then $d^{-}(v)=0$ for all $v\in Y,$ so Theorem \ref{thm 2.5} implies $%
\gamma _{oso}(D)=n.$
\end{proof}

\begin{prop} \label{prop 2.9} For the directed path  $P_n=(v_1,v_2,v_3,...,v_n)$ we have
$\gamma_{oso}(P_n)=\lceil \frac {2n}3\rceil.$
\end{prop}

\begin{proof}
 It is easy to check that the set $S=\{v_{i}:i\equiv 1,2\ ({ mod}\ 3)\}$ is an OSODS of $P_{n}.$ Therefore, $\gamma_{oso}(P_n)\leq\lceil \frac {2n}3\rceil.$

Now we will prove using induction on $n$ that $\gamma_{oso}(P_n)\geq\lceil \frac {2n}3\rceil.$ The result is obvious for $1\leq n\leq5.$ We assume that the result is true for any  directed path with less than $n$ vertices, and take   the directed path $P_{n}$ with $n\geq 6$ vertices. Let $S$ be a $\gamma_{oso}$-set of $P_{n}.$ Since $\gamma_{oso}(P_n)\leq\lceil \frac {2n}3\rceil,$ there is a vertex  $v_{i}\in V\setminus S$ with $i<n.$ Now, $P_n -v_iv_{i+1}\cong P_i \cup P_{n-i}.$ Consider the sets $S_1=S\cap V(P_i)$ and $S_2=S\cap V(P_{n-i}).$  Since $v_i$ does not out-dominate nor defend any vertex in $P_n,$ it follows that $S_1$ is an OSODS of $P_i$ and $S_2$ is an OSODS of $P_{n-i}$. From the induction hypothesis,   $\left|S_1\right|\geq\left\lceil \frac {2i}3\right\rceil$ and $ \left|S_2\right| \geq\left\lceil \frac {2(n-i)}3\right\rceil$. Hence $\left|S\right| \geq\left\lceil \frac {2(i+n-i)}3\right\rceil=\left\lceil \frac {2n}3\right\rceil,$ which implies that $\gamma_{iso}(P_n)\geq\left\lceil \frac {2n}3\right\rceil$ for every directed path  $P_n.$
\end{proof}

\begin{cor} \label{obs 2.14} For every tournament $T,$ $2\leq \gamma_{oso}(T)\leq \left\lceil\frac{2n}{3}\right\rceil.$
\end{cor}

\begin{proof}
The result follows because every tournament contains a directed hamiltonian path.
\end{proof}

\begin{prop}\label{prop 2.10}
 For the directed cycle  $C_n=(v_1,v_2,v_3,...,v_n,v_1)$ with $n\geq3$ we have
$\gamma_{oso}(C_n)=\lceil \frac {2n}3\rceil.$
\end{prop}
\begin{proof}
 From Observation \ref{obs 1.7} and Proposition \ref{prop 2.9},  $\gamma _{oso}(C_{n})~\leq $ $\gamma _{oso}(P_{n})= \left\lceil \frac{2n}{3}\right\rceil.$

For the converse, we proceed as in the proof of Proposition \ref{prop 2.9}: Let $C_{n}$ be the directed cycle with $n\geq 3$ vertices and let $S$ be a $\gamma_{oso}$-set of $C_{n}.$ Since $\gamma _{oso}(C_{n})\leq \left\lceil \frac{2n}{3}\right\rceil,$ there is a vertex  $v_{i}\in V\setminus S.$ We have that $C_n -v_iv_{i+1}\cong P_n,$ and since $v_i$ does not out-dominate nor defend any vertex in $C_n,$ it follows that $S$ is an OSODS of $P_n.$ Therefore, $\gamma _{oso}(C_{n})\geq \gamma _{oso}(P_{n})=\left\lceil \frac{2n}{3}\right\rceil.$
\end{proof}

\begin{prop}\label{prop 2.11}  Let $l(D)$ denote the length of a longest directed path in $D$.  Then
$\gamma_{oso}(D)\leq n-\lfloor\frac {l(D)+1}3\rfloor$ and the bound is sharp.
\end{prop}
\begin{proof}  Let $P=(v_0,v_1,v_2,...,v_k)$ be a longest directed path in $D.$   Let  $S$ be a minimum  OSODS of  $P$.
Clearly $S_1=S\cup (V(D)\setminus V(P))$ is an OSODS of $D,$ and hence $\gamma_{oso}(D)\leq |S_1|\leq n-\lfloor \frac {l(D)+1}3\rfloor$.
Equality holds obviously for directed paths, among other digraphs.
\end{proof}

The proof of the following proposition is similar.

\begin{prop} \label{prop 2.12} Let $c(D)$ denote the length of a longest directed cycle in $D$.
  Then $\gamma_{oso}(D)\leq n-\lfloor \frac {c(D)}3\rfloor$ and the bound is sharp.
 \end{prop}

\begin{prop} \label{thm 2.13} Let $T$ be a tournament of order $n\geq 3.$ If there exists $u\in V(T)$ such that $d^{-}(u)=0,$ then
 $2\leq \gamma_{oso}(T)\leq \left\lceil \log_{2}{(n-1)}  \right\rceil+1$.
\end{prop}

\begin{proof}  Take $u\in V(T)$ such that $d^{-}(u)=0.$ Then $u$ out-dominates the set $V(T)$ and  $u$ belongs to every OSODS of $T$. Let $T_1$ be the   subtournament obtained by deleting $u$ from $T.$ As in the proof of Fact 2.5 of \cite{NU},   since  $\sum\limits_{v\in V(T_1)}d^+(v)= \frac{(n-1)(n-2)}2 ,$ it follows that there exists a vertex $u_1$ in $T_1$ with $d^+(u_1)\geq \left\lceil\frac{n-2}{2}\right\rceil.$ Now, let $T_2=T_1\setminus N^+[u_1]$ and  let $u_2$ be a vertex of $T_2$ which out-dominates at least $\left\lceil\frac{|V(T_2)|}{2}\right\rceil$ vertices of $T_2.$

 By continuing this process
we obtain an out-dominating set $S$ of $T_1$ with  $|S|\leq\left\lceil \log_{2}{(n-1)}  \right\rceil$. Now $S\cup \{u\}$ is an OSODS of $T,$  and hence $\gamma_{oso}(T)\leq \left\lceil \log_2(n-1)\right\rceil+1.$
 \end{proof}

\section{Out-secure dominating sets}

\begin{prop}\label{obs 3.1}
Let $S$ be an OSDS  of a digraph  $D.$ Then a vertex $u\in N^{-}(v)\cap S$ defends a vertex $v\in V\setminus S$ if, and only if, $pn^{+}(u,S)\cup pn^{-}(u,S)\subseteq N^{+}[v]\cup N^{-}[v].$
\end{prop}
\begin{proof}
Let $D=(V,A)$ be a digraph with underlying graph $G$, and let $S$ be an OSDS
of $D.$ Take $v\in V\setminus S$ and $u\in N^{-}(v)\cap S.$

Suppose $(S\setminus \{u\})\cup \{v\}$ is dominating in $G,$ and take $w\in
(pn^{+}(u,S)\cup pn^{-}(u,S))\setminus \{v\}.$ Since $w$ is not adjacent (in $G$) to any
vertex in $S\setminus \{u\},$ it follows that $w$ is adjacent to $v$, that
is, $w\in N^{+}(v)\cup N^{-}(v).$

Now suppose $pn^{+}(u,S)\cup pn^{-}(u,S)\subseteq N^{+}[v]\cup N^{-}[v],$
and take a vertex $w\in V\setminus ((S\setminus \{u\})\cup \{v\}).$ If $w\in
pn^{+}(u,S)\cup pn^{-}(u,S)\cup \{u\},$ it is dominated by $v$. Otherwise, $w
$ has an in-neighbor or an out-neighbor in $S\setminus \{u\}.$ Therefore, $%
(S\setminus \{u\})\cup \{v\}$ is dominating in $G.$
\end{proof}

\begin{cor} \label{cor 3.2}
A set $S\subseteq V$ is an OSDS  of    $D$   if, and only if, for every $v\in V\setminus S,$ there exists $u\in N^{-}(v)\cap S$ such that   $pn^{+}(u,S)\cup pn^{-}(u,S)\subseteq N^{+}[v]\cup N^{-}[v].$
\end{cor}

\begin{prop}\label{prop 3.30}Let $G$ be a simple graph. Then there exists an orientation $D$ of $G$ such that $\gamma
_{s}(G)=\gamma _{os}(D).$
\end{prop}

\begin{proof}Let $G$ be a graph as in the hypothesis. From Proposition
\ref{prop 1.8}, $\gamma _{s}(G)\leq \gamma _{os}(D)$ for every orientation $D$ of $G.$
Conversely, let $S$ be a minimum secure set of $G,$ and consider the
following orientation $D$ of $G$: For every two adjacent vertices $u\in S$
and $v\in V\setminus S,$ give the orientation $uv$ to their common edge;
edges between vertices of $S$ and edges between vertices of $V\setminus S$
can be oriented arbitrarily. Then $S$ is an OSDS of $D$: In $G,$ for every $%
v\in V\setminus S$ there exists $u\in S\cap N(v)$ such that $(S\setminus
\{u\})\cup \{v\}$ is dominating, and $u\in N^{-}(v)$ in $D.$
\end{proof}

\begin{cor} \label{cor 3.32} Let $G$ be a bipartite simple graph with bipartition $(X,Y)$ such that for every $v\in
Y,~d(v)\geq 2.$ Then there exists an orientation $D$ of $G$ such that $X$ is
an OSDS of $D.$
\end{cor}

\begin{proof} The result follows because $X$ is an SDS of $G,$ since every
vertex $v\in Y$ has at least two neighbors in $X.$ As in Proposition \ref{prop 3.30}, we
give to $G$ the orientation $D$ in which $d^{-}(u)=0$ for every $u\in X.$
\end{proof}

\begin{cor} \label{cor 3.33} Let $G$ be a  simple $C_5$-free graph with $\delta\geq2.$  Then there exists an orientation $D$ of $G$ such
that $\gamma _{os}(D)\leq \frac{n}{2}.$
\end{cor}
\begin{proof} The result follows from Proposition \ref{prop 3.30} and a theorem appearing in \cite{A0}.
\end{proof}

\begin{prop}\label{prop 3.31} Let $G=(V,E)$ be a simple graph, and let $I$ be an
independent set of vertices (i.e. no two vertices of $I$ are adjacent) such
that for every vertex $v\in I,~d(v)\geq 2.$ Then there exists an orientation
$D$ of $G$ such that $\gamma _{os}(D)\leq |V\setminus I|.$
\end{prop}

\begin{proof} Let $G=(V,E)$ and $I$ be as in the hypothesis. We give to $G$
the following orientation $D$: For every $v\in I$ and every $u\in N(v),$ we
asign the arc $uv.$ All other edges are oriented arbitrarily. Then $%
V\setminus I$ is an OSDS of $D,$ since for every $v\in I$ there exists $u\in
(V\setminus I)\cap N^{-}(v)$ such that $(S\setminus \{u\})\cup \{v\}$ is
dominating.
\end{proof}

\begin{theorem} \label{thm 3.3}
Let $D$ be any digraph without symmetric arcs. Then    $$ \gamma_{os}(D)\leq \gamma^{+}(D)+ \gamma^{-}(D).$$
\end{theorem}

\begin{proof}
Let $D$ be a digraph without symmetric arcs and let $G$ be the underlying
undirected graph of $D.$ Let $S^{+}$ and $S^{-}$ be minimum out- and
in-dominating sets of $D,$ respectively. Then $S=S^{+}\cup S^{-}$ is an OSDS
of $D$: Take a vertex $v\in V\setminus S$ and a vertex $u\in N^{-}(v)\cap S.$
The set $(S\setminus \{u\})\cup \{v\}$ is dominating in $G,$ since $v$
dominates $u,$ every vertex in $N^{+}(u)\cap (V\setminus S)$ has an
out-neighbor in $S\setminus \{u\},$ and every vertex in $N^{-}(u)\cap
(V\setminus S)$ has an in-neighbor in $S\setminus \{u\}.$
 \end{proof}

\begin{theorem}
For any digraph $D$ with $n\geq2,$  $\gamma_{os}(D)\leq n-1.$ If $D$ is connected, equality holds if, and only if, $D$ is the directed cycle $C_3$ or the underlying graph $G$ of $D$ is a star.
\end{theorem}
\begin{proof}
Since $D=(V,A)$ has no isolated vertices, there exists at least one vertex $v\in V$ such that $d^-(v)\geq 1$. Therefore, $V\setminus \{v\}$ is an OSDS of $D.$

Let $D=(V,A)$ be a digraph such that $\gamma_{os} (D)=n-1.$ It is easy to check the result for $n\leq3.$ Moreover, from Proposition \ref{prop 1.8}, $ \gamma_{s}(G)\leq\gamma_{os}(D);$ therefore, Proposition 10 of \cite{E3} implies that for every digraph $D$ such that $G$ is a star, $\gamma_{os}(D)=n-1.$

 Now assume $n\geq4.$ Then there is a vertex $v\in V$ such that all vertices in $V\setminus \{v\}$ are adjacent to $v.$ Otherwise, there exists $\{u,v\}\subseteq V$ such that $u$ and  $v$ are not adjacent. Since $D$ is connected,  there exists $w\in V$ such that $wu\in A$ or $uw\in A.$ Since $n\geq4,$ then there exists $x\in V\setminus \{w\}$ such that $vx\in A$ or $xv\in A.$ It follows that $V\setminus \{w,x\},$ $V\setminus \{x,u\},$  $V\setminus \{v,w\},$ or $V\setminus \{v,u\},$ is a an OSDS of $D$ (for example, if $uw\in A$ and $vx\in A$, then $V\setminus \{w,x\}$ is an OSDS), which is a contradiction. Therefore, there exists one vertex $v$  such that all  vertices in $V\setminus \{v\}$ are adjacent to $v.$

Now we show that   all  vertices in $V\setminus \{v\}$ are independent. Consider  $\{u,w\} \subseteq V\setminus \{v\}$ such that $uw\in A.$  Since $n\geq4,$ there exists $x\in V\setminus \{u,w\}$ such that $vx\in A$ or $xv\in A.$ This implies that $V\setminus \{x,w\}$ or  $V\setminus \{v,w\}$ is an OSDS of $D,$ which is a contradiction.
\end{proof}

\begin{prop} \label{prop 3.5} Let $P_n=(v_1,v_2,v_3,...,v_n)$ be a directed path.    Then
$\gamma_{os}(P_n)=\lceil \frac {n}2\rceil.$
\end{prop}

\begin{proof}  From Proposition \ref{prop 1.8} we have $\gamma^{+}(D)\leq \gamma_{os}(D),$ so it follows from Observation \ref{obs 1.9} that    $\lceil \frac {n}2\rceil\leq \gamma_{os}(P_n).$  Further, it is easy to check that $S=\{v_i: i\equiv 1\ {\rm(mod }$2$)\}$  is an OSDS of $P_n.$
\end{proof}

\begin{prop}\label{prop 3.6}
  Let $C_n=(v_1,v_2,v_3,...,v_n,v_1)$ be a directed cycle with $n\geq 3.$    Then
$\gamma_{os}(C_n)=\lceil \frac {n}2\rceil.$
\end{prop}
\begin{proof}
From Observations \ref{obs 1.7} and \ref{obs 1.9}, and Proposition \ref{prop 3.5} the result follows.
\end{proof}

\begin{prop} \label{prop 3.7}
 Let $l(D)$ denote the length of a longest directed path in $D$.  Then
$\gamma_{os}(D)\leq n-\lfloor\frac {l(D)+1}2\rfloor$ and the bound is sharp.
\end{prop}

\begin{proof}
 The result follows from Proposition \ref{prop 3.5} in a similar way as Proposition \ref{prop 2.11} follows from Proposition \ref{prop 2.9}.
\end{proof}

\begin{prop} \label{prop 3.8}
 Let $c(D)$ denote the length of a longest directed cycle in $D$.
  Then $\gamma_{os}(D)\leq n-\lfloor \frac {c(D)}2\rfloor$ and the bound is sharp.
 \end{prop}
 \begin{proof}
  The result follows from Proposition \ref{prop 3.6} in a similar way as Proposition \ref{prop 2.12} follows from Proposition \ref{prop 2.10}.
\end{proof}

\begin{prop} Let $T$ be a tournament. Then $\gamma_{os}(T)=1$ if, and only if, there exist $u\in V$ such that $d^-(u)=0.$
\end{prop}

\begin{proof} Suppose $\gamma_{os}(T)=1$ and let $S=\{u\}$ be an OSDS of $T.$ Let $v\in V\setminus \{u\}.$  If $v$ is defended by $u,$ Proposition \ref{obs 3.1} implies that $u\in N^-(v).$ Hence $d^-(u)=0.$ Conversely, if $d^-(u)=0,$ then $\{u\}$ is an OSDS of $T$ and hence $\gamma_{os}(T)=1.$
\end{proof}

\begin{cor}\label{cor 3.9} If $T$ is a transitive tournament, then $\gamma_{os}(T)=1.$
\end{cor}

\begin{prop} \label{thm 3.10} Let $T$ be a tournament of order $n\geq 2,$  then
 $\gamma_{os}(T)\leq \left\lceil \log_{2}{n}  \right\rceil.$
\end{prop}

\begin{proof} From Fact 2.5 of \cite{NU}, there is an out-dominating set $S$ such that $|S|\leq\left\lceil \log_{2}{n}  \right\rceil.$  It is straightforward that $S$ is an OSDS of $T,$  and hence $\gamma_{os}(T)\leq \left\lceil \log_2n\right\rceil.$
 \end{proof}

\begin{dfn}{\rm\cite{char}} Let $D=(V, A)$ be a digraph.  A subset $S$ of $V$ is called a twin dominating set of $D$
if for every vertex $v\in V\setminus S$, there exist two vertices $u_1,u_2\in S$ $(u_1$ and $u_2$ may possibly coincide$)$ such that
$(v, u_1)$ and $(u_2, v)$ are arcs in $D$. The minimum cardinality of a twin dominating set in $D$ is called the twin domination number of $D$ and is denoted by $\gamma^*(D).$
\end{dfn}

\begin{prop}
For any digraph $D$ without symmetric arcs,  $\gamma_{os}(T)\leq \gamma^*(D).$
\end{prop}

\begin{proof}
Let $S$ be a minimum twin dominating set of $D.$ Then every vertex  $v\in V\setminus S$ is out-dominated by at least one vertex $u\in S.$ It follows that  $u$  defends  $v$ and $(S\setminus \{u\})\cup \{v\}$ is a dominating set of $G,$ the underlying undirected graph of $D,$   since every vertex in $ (N^{-}(u)\cup N^{+}(u))\cap (V\setminus S)$ is adjacent to at least one vertex in $S\setminus \{u\}.$
\end{proof}
\begin{theorem} \label{thm 3.11}
{\rm\cite{char}} Let $D$ be a digraph  with $\delta^0> 0,$  then
   $\gamma^*(D)\leq \lfloor\frac{2n}3\rfloor.$
\end{theorem}

\begin{cor}  Let $D$ be a digraph with $\delta^0> 0,$  then
   $\gamma_{os}(D)\leq \lfloor\frac{2n}3\rfloor.$
\end{cor}

\section{Secure out-dominating sets}

\begin{prop}\label{prop 4.1}
 Let $S$ be an SODS of a digraph $D.$ A vertex $u\in
N^{+}(v)\cap S$ defends a vertex $v\in V\setminus S$ if, and only if, $%
pn^{+}(u,S)\subseteq N^{+}[v].$
\end{prop}
\begin{proof}
If there is a vertex $w\in pn^{+}(u,S)\setminus N^{+}[v]$, then $(S\setminus
\{u\})\cup \{v\}$ does not out-dominate $w.$ The converse is obvious.
\end{proof}

\begin{prop}\label{prop 4.2}
 Let $S$ be an SODS of a digraph $D.$ A vertex $u\in
N^{-}(v)\cap S$ defends a vertex $v\in V\setminus S$ if, and only if, $N^{-}(u)\cap S\neq \emptyset ,$ and  $pn^{+}(u,S)\subseteq N^{+}[v].$
\end{prop}

\begin{proof}
The proof of this proposition is identical to that of Proposition \ref{prop 2.1}.
\end{proof}

\begin{cor}\label{cor 4.3}
 A set $S\subseteq V$ is an SODS of $D$ if, and only if, $S$
is out-dominating and for every $v\in V\setminus S$ one of the two following
conditions hold:

\begin{enumerate}
\item There exists a vertex $u\in N^{+}(v)\cap S$ such that $%
pn^{+}(u,S)\subseteq N^{+}[v].$

\item There exists a vertex $u\in N^{-}(v)\cap S$ such that $%
pn^{+}(u,S)\subseteq N^{+}[v]$ and $N^{-}(u)\cap S\neq \emptyset .$
\end{enumerate}
\end{cor}

\begin{theorem}\label{prop 5.3}
For every digraph $D$ without symmetric arcs, $\frac{n+1}{\Delta ^{+}+1}\leq
\gamma _{so}(D),$ and the bound is sharp even for $\gamma _{iso}(D).$
\begin{proof}
Let $S$ be an SODS of $D$. It is clear that $n\leq (\Delta
^{+}+1)\gamma _{so}(D).$ If there are $\{u,w\}\subseteq S$ such that $uw\in
A,$ we have that $n\leq (\Delta ^{+}+1)\gamma _{so}(D)-1;$ similarly, if
there are $\{u,w\}\subseteq S$ such that $N^{+}(u)\cap N^{+}(w)\neq
\emptyset ,$ then $n\leq (\Delta ^{+}+1)\gamma _{so}(D)-1.$ Therefore, we
can assume that for every $u\in S,$ $N^{+}(u)=pn^{+}(u,S).$ If there exists $%
u\in S$ such that $d^{+}(u)<\Delta ^{+},$ then we have as well $n\leq
(\Delta ^{+}+1)\gamma _{iso}(D)-1.$ Assume then that every $u\in S$ has
out-degree $\Delta ^{+}$, and that all its out-neighbors are in $pn^{+}(u,S).
$ Take $v\in V\setminus S;$ since $S$ is an SODS, there is a vertex $w\in
(N^{+}(v)\cup N^{-}(v))\cap S$ such that $(S\setminus \{w\})\cup \{v\}$ is
out-dominating. From Proposition \ref{prop 4.1} and Proposition \ref{prop 4.2}, we have that $pn^{+}(w,S)\subseteq
N^{+}[v].$ If $w\in N^{-}(v),$ Proposition \ref{prop 4.2} implies that $N^{-}(w)\cap
S\neq \emptyset $, so there is a vertex $w^{\prime }\in S$ such that $%
N^{+}(w^{\prime })\neq pn^{+}(w^{\prime },S),$ a contradiction. If $w\in
N^{+}(v),$ then $|N^{+}(v)|~\geq |pn^{+}(w,S)|~+1=\Delta ^{+}+1,$ which is
also a contradiction. Therefore, for every SODS of $D$ it holds that there
exists $u\in S$ such that either $N^{+}(u)\neq pn^{+}(u,S)$ or $%
d^{+}(u)<\Delta ^{+},$ which implies that $n\leq (\Delta ^{+}+1)\gamma
_{so}(D)-1.$

To show that the bound is sharp even for $\gamma _{iso}$, consider the
digraph $D=(V,A),$ where $V=\{w,u_{1},...,u_{k},v_{1},...v_{k}\}$ and $%
A=\{u_{i}v_{i}:i\in \{1,...,k\}\}\cup \{v_{i}w:i\in \{1,...,k\}\}.$ It is
clear that for every SODS $S$, the set $B=\{u_{1},...,u_{k}\}\subset S,$
since those vertices are not out-dominated by anyone. This is not enough,
since the vertices in $V\setminus B$ are not defended. However, $%
S=\{w,u_{1},...,u_{k}\}$ is a minimum ISODS, since $w$ defends every vertex
in $V\setminus S$. Then we have $2k+1=n=(\Delta ^{+}+1)\gamma _{iso}(D)-1.$
\end{proof}
\end{theorem}

Since every OSODS of a digraph $D$ is an SODS of $D,$ the bound stated in Theorem \ref{prop 5.3} applies to OSODSs as well. The bound in Theorem \ref{prop 5.3} is better than that in Theorem \ref{thm 2.3} if, and only if, $\Delta ^{+}>\frac{n-1}{2}.$

\begin{theorem}\label{thm 4.10}
 Let $D$ be any digraph. Then $\gamma _{so}(D)=n$ if, and
only if, for every $u\in V(D),$ $d^{+}(u)=0$ or $d^{-}(u)=0$.
\begin{proof}
The proof of this result is almost identical to that of Theorem \ref{thm 2.5}.
\end{proof}
\end{theorem}

\begin{cor}\label{cor 4.11}
A nontrivial graph $G$ admits an orientation $D$ such that $\gamma
_{so}(D)=n$ if, and only if, $G$ is a bipartite graph.
\end{cor}

\begin{proof}
Proof is similar to that of Corollary \ref{thm 2.6}, using Theorem \ref{thm 4.10}.
\end{proof}

\begin{prop}\label{prop 4.6}
 Let $P_n=(v_1,v_2,v_3,...,v_n)$ be a directed path. Then $\gamma _{so}(P_{n})=\lceil \frac{3n}{5}\rceil .$
\end{prop}

\begin{proof}

It is clear that $S=\{v_{i}:i\equiv 1,3,4\ ({ mod}\ 5)\}$ is an SODS of $P_n$ if $n\equiv 0,1,3,4\ ({mod}\ 5),$ and $S\cup\{v_n\}$ is an SODS of $P_n$ if $n\equiv 2\ ({mod}\ 5)$. Therefore, $\gamma _{so}(P_{n})\leq\lceil \frac{3n}{5}\rceil.$

We will prove the converse by induction on $n.$ It can be checked that $\gamma _{so}(C_{n})\geq\lceil \frac{3n}{5}\rceil $
for $1\leq n\leq 22,$ and observe that if $n>22$ then $\lceil \frac{3n}{5}\rceil <\lceil \frac{2n}{3}\rceil-1$. Consider $P_n=(V,A)$ such that $n\geq 23,$ and let $S$ be a $\gamma _{so}$-set of $P_{n}.$ From Proposition  \ref{prop 4.1} and Proposition  \ref{prop 4.2}, it follows that if a vertex $u\in S$ defends some vertex $v\in V\setminus S,$ then $(N^{-}(u)\cup N^{+}(u))\cap S\neq \emptyset$ (because all vertices have in-degree at most $1$).

Therefore, if at most one vertex $u\in S$ does not defend any vertex $v\in V\setminus S,$ then $|S|~\geq $ $\lceil \frac{2n}{3}\rceil,$ which is a contradiction. This implies that there are two vertices in $S$  not defending any vertex in $V\setminus S.$ Therefore, again from Proposition  \ref{prop 4.1} and Proposition  \ref{prop 4.2}, there is a vertex $v_{i+1}\in S$ with $1\leq i\leq n-1$ such that $(N^{-}(v_{i+1})\cup N^{+}(v_{i+1}))\cap S=\emptyset .$

Consider the digraph $P_{n}-v_{i}v_{i+1}\cong P_i\cup P_{n-i}$ and the sets $S_1=S\cap V(P_i)$ and $S_2=S\cap V(P_{n-i}).$ Then $S_1$ is an SODS of $P_i$ and $S_2$ is an SODS of $P_{n-i},$ since $v_{i+1}$ does not defend any vertex of $V\setminus S$ in $P_{n},$ and it out-dominates $v_{i+2}$ (if it exists) both in $P_{n}$ and in $P_{n}-v_{i}v_{i+1}.$ From the induction hypothesis,   $\left|S_1\right|\geq\left\lceil \frac {3i}5\right\rceil$ and $ \left|S_2\right| \geq\left\lceil \frac {3(n-i)}5\right\rceil$. Hence $\left|S\right| \geq\left\lceil \frac {3(i+n-i)}5\right\rceil=\left\lceil \frac {3n}5\right\rceil,$ which implies that $\gamma_{iso}(P_n)\geq\left\lceil \frac {3n}5\right\rceil$ for every directed path  $P_n.$
\end{proof}

\begin{prop}\label{prop 4.7}
 Let $C_n=(v_1,v_2,v_3,...,v_n,v_1)$ be a directed cycle with $n\geq 3$ vertices. Then $\gamma _{so}(C_{n})=\lceil \frac{3n}{5}\rceil .$
\begin{proof}
By Observation \ref{obs 1.7}  and Proposition \ref{prop 4.6}, $\gamma
_{so}(C_{n})\leq \gamma _{so}(P_{n})=\lceil \frac{3n}{5}\rceil .$\newline
For the converse, we proceed as in Proposition \ref{prop 4.6}. It is easy to check that $\gamma _{so}(C_{n})=\lceil \frac{3n}{5}\rceil $
for $3\leq n\leq 12.$ Moreover, for $n>12$ we have that $\lceil \frac{3n}{5}%
\rceil <\lceil \frac{2n}{3}\rceil $. Let $C_{n}=(V,A)$ be a directed cycle
with $n>12,$ and let $S$ be a $\gamma _{so}$-set of $C_{n}.$ From
Proposition  \ref{prop 4.1} and Proposition  \ref{prop 4.2}, it follows that if
a vertex $u\in S$ defends some vertex $v\in V\setminus S,$ then $(N^{-}(u)\cup N^{+}(u))\cap S\neq \emptyset$ (because all vertices have in-degree $1$).

Therefore, if every vertex $u\in S$ defends a vertex $v\in V\setminus S,$ then $%
|S|~\geq $ $\lceil \frac{2n}{3}\rceil ,$ which is a contradiction. This
implies that there is a vertex $v_{i}\in S$ not defending any vertex in $V\setminus S,$
and then, again from Proposition  \ref{prop 4.1} and Proposition  \ref{prop 4.2}, we have $(N^{-}(v_{i})\cup N^{+}(v_{i}))\cap S=\emptyset .$

Consider the path $C_{n}-v_{i-1}v_{i}\cong P_{n}.$ Then $S$ is an SODS of $%
C_{n}-v_{i-1}v_{i},$ since $v_{i}$ does not defend any vertex of $V\setminus
S$ in $C_{n},$ and it out-dominates $v_{i+1}$ both in $C_{n}$ and in $%
C_{n}-v_{i-1}v_{i}.$ Moreover, $C_{n}-v_{i-1}v_{i}\cong P_{n}$ implies that $%
\gamma _{so}(C_{n}) = |S|~\geq \gamma _{so}(P_{n})=\lceil \frac{3n}{5}%
\rceil .$  Therefore, for every cycle $C_{n}$ with $n\geq 3,$ $\gamma
_{so}(C_{n})= \lceil \frac{3n}{5}\rceil .$
\end{proof}
\end{prop}

\begin{prop}\label{prop 4.8}
Let $l(D)$ denote the length of a longest directed path in $D$. Then $\gamma
_{so}(D)\leq n-\lfloor \frac{2l(D)+2}{5}\rfloor $ and the bound is sharp.
\end{prop}

\begin{proof}
 The result follows from Proposition \ref{prop 4.6} in a similar way as Proposition \ref{prop 2.11} follows from Proposition \ref{prop 2.9}.
\end{proof}

\begin{prop}\label{prop 4.9}
Let $c(D)\geq 3$ denote the length of a longest directed cycle in $D$. Then $%
\gamma _{so}(D)\leq n-\lfloor \frac{2c(D}{5}\rfloor $ and the bound is
sharp.
\end{prop}
 \begin{proof}
  The result follows from Proposition \ref{prop 4.7} in a similar way as Proposition \ref{prop 2.12} follows from Proposition \ref{prop 2.10}.
\end{proof}

\begin{theorem}\label{thm 4.12}
 Let $T$ be a tournament. Then $\gamma _{so}(T)\leq \gamma ^{+}(T)+1.$

\begin{proof}
Let $S$ be the minimum out-dominating set of $T$ and $v\in V\setminus S,$
then clearly $S\cup \{v\}$ is an SODS of $T,$ since for every vertex $u\in
V\setminus (S\cup \{v\})$ the set $S\cup \{u\}$ is an out-dominating set of $%
T.$ Therefore, $\gamma _{so}(T)\leq \gamma ^{+}(T)+1$.

Conversely, in a tournament we have at most one vertex $v$ such that $%
d^{+}(v)=n-1.$ Then $\{v\}$ is the only out-dominating set of cardinality $1$%
. However, $v$ does not defend any vertex in $V\setminus \{v\}$. This
implies $2\leq \gamma _{so}(T).$
\end{proof}
\end{theorem}

\begin{cor}
For any tournament $T$ with $n\geq 2,$ $\gamma _{so}(T)\leq \lceil \log _{2}{n}\rceil +1$.
\end{cor}

\begin{proof}
The result follows because $\gamma^+(T)\leq \lceil \log _{2}{n}\rceil,$ as shown in Fact 2.5 of {\rm\cite{NU}}.
\end{proof}

\section{In-secure out-dominating sets}

\begin{prop}\label{prop 5.1}
Let $S$ be an ISODS of a digraph $D.$ A vertex $u\in N^{+}(v)\cap S$ defends
a vertex $v\in V\setminus S$ if, and only if, $pn^{+}(u,S)\subseteq
N^{+}[v]. $
\begin{proof}
If $w\in pn^{+}(u,S),$ and $vw\notin A,$ then $(S\setminus \{u\})\cup \{v\}$
does not dominate $w$. The converse is obvious.
\end{proof}
\end{prop}

\begin{cor}\label{cor 5.2}
A set $S\subseteq V$ is an ISODS if, and only if, for every $v\in V\setminus
S$ there exists $u\in N^{+}(v)\cap S$ such that $pn^{+}(u,S)\subseteq
N^{+}[v].$
\end{cor}

\begin{theorem}\label{prop 5.4}
For every digraph $D,$  $\gamma _{iso}(D)\leq n-\delta ^{-},$ and the bound is sharp.

\begin{proof}
The result is trivial if $\delta^{-}=0.$ Hence we assume that $\delta^{-}>0.$ Take $u\in V(D)$ and $B\subseteq N^{-}(u)$ such that $|B|=\delta^{-}$. Then $S=V\setminus B$ is an ISOSD: every $v\in N^{-}(u)$ has at least one in-neighbor in $S,$ so $S
$ is out-dominating. Also, for every $v\in N^{-}(u),$ $(S\setminus
\{u\})\cup \{v\}$ is out-dominating, because $v$ out-dominates $u.$

To show that the bound is sharp, we observe that for the directed $4$-cycle $C_4,$ $\gamma _{iso}(C_4)=3.$
\end{proof}
\end{theorem}

\begin{obs}\label{obs 5.5}
Let $D=(V,A)$ be a digraph such that there exists $\{u,v\}\subseteq V$ with $%
N^{+}(u)\setminus \{v\}=N^{-}(v)\setminus \{u\}=V\setminus\{u,v\}.$ Then $\gamma _{iso}(D)=2.$
\end{obs}

\begin{cor}
For every transitive tournament $T,~\gamma _{iso}(T)=2.$
\end{cor}

\begin{theorem}\label{prop 5.6}
Let $D$ be any digraph. Then $\gamma _{iso}(D)=n$ if, and only if, for every
$u\in V(D),$ $d^{+}(u)=0$ or $d^{-}(u)=0$.
\begin{proof}
All vertices with in-degree zero must belong to every out-dominating set. Moreover, any vertex with out-degree zero cannot be defended, so all vertices in $V$ must belong to every ISODS.
\end{proof}
\end{theorem}

\begin{cor}
A nontrivial graph $G$ admits an orientation $D$ such that $\gamma _{iso}(D)=n$
if, and only if, $G$ is a bipartite graph.
\end{cor}
\begin{proof}
The proof is similar to that of Corollary \ref{thm 2.6}, using Proposition \ref{prop 5.6}.
\end{proof}

\begin{prop} \label{prop 5.8}
For the directed path  $P_n=(v_1,v_2,v_3,...,v_n)$ we have
$\gamma_{iso}(P_n)=\lceil \frac {2n}3\rceil.$
\end{prop}

\begin{proof}
 It is easy to check that the set $S=\{v_{i}:i\equiv 0,1\ ({ mod}\ 3)\}$ is an ISODS of $P_{n}$ if
    $n\equiv 0,1\ ({mod}\ 3),$ and $S\cup \{v_{n}\}$ is an ISODS
of $P_{n}$ if $n\equiv 2\ ({mod}\ 3).$ Therefore, $\gamma_{iso}(P_n)\leq\left\lceil \frac {2n}3\right\rceil.$

 We will prove by induction on $n$  that $\gamma_{iso}(P_n)\geq\left\lceil \frac {2n}3\right\rceil$ for every directed path $P_n$. If $n=2,3$, clearly $\gamma_{iso}(P_n)=2\geq\left\lceil \frac {2n}3\right\rceil$. We assume that the result is true for any  directed path with less than $n$ vertices, and let $S$ be a $\gamma_{iso}$-set of $P_n.$ Since $n\geq4,$ then at least two vertices of $S$ are adjacent. Otherwise, take $v_i\in S$ with
      $1<i<n.$   From Proposition \ref{prop 5.1}, $v_{i}$ cannot defend $v_{i-1},$ so $v_{i-1}$ is not defended. Therefore, if $n\geq4,$ there exist at least two adjacent vertices in $S.$

      Suppose  $v_i,\ v_{i+1}\ (1\leq i\leq n-1)$ are adjacent vertices in $S.$  Then $P_n -v_iv_{i+1}\cong P_i \cup P_{n-i}.$   Let $S_1$ be an ISODS of $P_i$ and $S_2$ be an ISODS of $P_{n-i}$.  Clearly $v_{i}\in S_1$ and $v_{i+1} \in S_2,$ and hence $ S=S_1\cup S_2$ is an ISODS of $P_n$.  Also by the induction hypothesis,   $\left|S_1\right|\geq\left\lceil \frac {2i}3\right\rceil$ and $ \left|S_2\right| \geq\left\lceil \frac {2(n-i)}3\right\rceil$.
    Hence $\left|S\right| \geq\left\lceil \frac {2(i+n-i)}3\right\rceil\geq\left\lceil \frac {2n}3\right\rceil$. So $\gamma_{iso}(P_n)\geq\left\lceil \frac {2n}3\right\rceil$ for every directed path  $P_n.$

\end{proof}

\begin{prop}\label{prop 5.9}
 For the directed cycle  $C_n=(v_1,v_2,v_3,...,v_n,v_1)$ with $n\geq3$ we have
$\gamma_{iso}(C_n)=\lceil \frac {2n}3\rceil.$
\end{prop}
\begin{proof} By   Observation \ref{obs 1.7} and Proposition \ref{prop 5.8}, $\gamma_{iso}(C_n)\leq\gamma_{iso}(P_n)= \lceil \frac {2n}3\rceil.$

It is easy to check that for $3\leq n\leq4,$ $\gamma_{iso}(C_n)\geq\lceil \frac {2n}3\rceil.$ For $n\geq5$, following a reasoning similar to that of Proposition \ref{prop 5.8}, in every ISODS of $C_n$ there are at least two adjacent vertices. Let $S$ be a $\gamma_{iso}$-set of $C_n$, and let $\{v_i,v_{i+1}\}\subseteq S.$ Notice that $S$ is an ISODS of the path $C_n-v_iv_{i+1}\cong P_n,$ since $v_{i+1}$ out-dominates $v_{i+2},$ $v_{i}$ defends $v_{i-1},$ and all other adjacencies are not altered. Therefore, $\gamma_{iso}(C_n)\geq \gamma_{iso}(P_n)=\lceil \frac {2n}3\rceil.$
\end{proof}

\begin{prop} \label{prop 5.10} Let $l(D)$ denote the length of a longest directed path in $D$.  Then
$\gamma_{iso}(D)\leq n-\lfloor\frac {l(D)+1}3\rfloor$ and the bound is sharp.
\end{prop}
\begin{proof}
 The result follows from Proposition \ref{prop 5.8} in a similar way as Proposition \ref{prop 2.11} follows from Proposition \ref{prop 2.9}.
\end{proof}

\begin{prop} \label{prop 5.11} Let $c(D)$ denote the length of a longest directed cycle in $D$.
  Then $\gamma_{iso}(D)\leq n-\lfloor \frac {c(D)}3\rfloor$ and the bound is sharp.
 \end{prop}
\begin{proof}
 The result follows from Proposition \ref{prop 5.9} in a similar way as Proposition \ref{prop 2.12} follows from Proposition \ref{prop 2.10}.
\end{proof}

\section{Conclusions and scope}

In this work we introduced four ways of extending secure dominating sets to digraphs, each with mathematical interest and several applications. There are a lot of questions remaining open. We write here some of the most interesting:

\begin{prb} Characterize digraphs for which equality holds in the inequalities stated in Proposition \ref {prop 1.8}.
\end{prb}

\begin{prb} Characterize digraphs for which equality holds in the inequalities stated in Propositions \ref {prop 2.11}, \ref {prop 2.12}, \ref {prop 3.7}, \ref {prop 3.8}, \ref {prop 4.8}, \ref {prop 4.9}, \ref {prop 5.10}, \ref {prop 5.11}.
\end{prb}

\begin{prb} Characterize digraphs for which equality holds in the inequalities stated in Theorems \ref {thm 2.3}, \ref {thm 2.4}, \ref {thm 3.3}, and Propositions \ref {prop 5.3}, \ref {prop 5.4}.
\end{prb}

We also have the following conjecture:

\begin{con}
 If $D$ is a digraph of order $n$ with $\delta^{0}\geq 1,$ then $\gamma _{oso}(D)\leq \lceil\frac{2n}{3}\rceil $ and $\gamma _{iso}(D)\leq \lceil\frac{2n}{3}\rceil .$
\end{con}

It may be easier to prove the result for $\gamma _{so}(D).$\\

It is not difficult to see that for most graphs $G$, different orientations of $G$  result in different cardinalities for a minimum OSODS, OSDS, SODS, and ISODS. This suggests the following definitions:
For a graph $G,$ the lower orientable  out-secure out-domination number $dom_{oso}(G)$ of $G$ is  defined by

$$ dom_{oso}(G)=min\{\gamma_{oso}(D)\ |\ D\ is\ an\ orientation\ of\ G \},$$
while the upper orientable out-secure out-domination number $DOM_{oso}(G)$ of  $G$ is defined by

$$ DOM_{oso}(G)=max\{\gamma_{oso}(D)\ |\ D\ is\ an\ orientation\ of\ G \}.$$

The parameters $dom_{os}(G),\ dom_{so}(G),$ and $dom_{iso}(G)$ are defined in a similar way, as are its DOM counterparts. There are several interesting questions arising from these definitions. In particular, we have the following open problem:
\begin{prb}  Let $G$ be a graph with $dom_{os}(G)=a$ and $DOM_{os}(G)=b.$ If $c$ is an integer with $a\leq c\leq b,$ is there an orientation $D$ of $G$ such that $\gamma_{os}(D)=c$?
\end{prb}
The same problem is worth considering for OSODS, SODS, and ISODS. We believe that the answer is affirmative for OSDS, while for the other parameters we have no guess.

\end{document}